\documentclass[reqno, 11pt]{amsart}

\makeatletter
\let\origsection=\section \def\section{\@ifstar{\origsection*}{\mysection}} 
\def\mysection{\@startsection{section}{1}\z@{.7\linespacing\@plus\linespacing}{.5\linespacing}{\normalfont\scshape\centering\S}}
\makeatother    

\usepackage{geometry}
\numberwithin{equation}{section}
\numberwithin{figure}{section}

\usepackage{xcolor}
\usepackage{hyperref}
\hypersetup{
    unicode,
    colorlinks,
    linkcolor={red!60!black},
    citecolor={green!60!black},
    urlcolor={blue!60!black}
}

\usepackage{amsfonts,amsthm,amssymb,amsmath}
\usepackage[T1]{fontenc}
\usepackage{microtype}
\usepackage{graphicx}
\usepackage{tikz}
\usepackage{tkz-berge}
\usepackage[usenames,dvipsnames]{pstricks}
\usepackage{epsfig}
\usepackage{verbatim}
\usepackage{scalerel}

\usepackage{enumerate,enumitem}
\usepackage{thmtools, thm-restate}
\usetikzlibrary{decorations.pathmorphing, arrows, calc, matrix, arrows.meta}

\usepackage{mathtools}
\usepackage{subcaption}

\newtheorem{theorem}{Theorem}

\newcommand{\script}{\mathcal}
\newcommand{\parentheses}[1]{{\left( {#1} \right)}}

\newcommand{\p}{\parentheses}

\newcommand{\Set}[1]{{\left\lbrace {#1} \right\rbrace}}

\def\set#1:#2{\Set{{#1} \colon {#2}}}

\newcommand{\N}{\mathbb{N}}

\renewcommand{\triangleleft}{\vartriangleleft}
\renewcommand{\leq}{\leqslant}
\renewcommand{\geq}{\geqslant}
\renewcommand{\preceq}{\preccurlyeq}
\renewcommand{\rho}{\varrho}
\renewcommand{\subset}{\subseteq}

\newcommand{\nottriangleleft}{\not\kern-1pt\mathrel{\triangleleft}}

\begin{document}
\title{A note on minor antichains of uncountable graphs}   

\author{Max Pitz}
\address{Hamburg University, Department of Mathematics, Bundesstra\ss e 55 (Geomatikum), 20146 Hamburg, Germany}
\email{max.pitz@uni-hamburg.de}



\begin{abstract}
A simplified construction is presented for Komj\'ath's result  that for every uncountable cardinal $\kappa$, there are $2^\kappa$ graphs of size $\kappa$ none of them being a minor of another.
\end{abstract}

\maketitle

\section{Introduction}
The famous Robertson-Seymour Theorem asserts that the class of finite graphs is well-quasi-ordered under the minor relation $\preceq$: For every sequence $G_1,G_2,\ldots$ of finite graphs there are indices $i < j$ such that $G_i \preceq G_j$.\footnote{Recall that a graph $H$ is a \emph{minor} of another graph $G$, written $H \preceq G$, if to every vertex $x \in H$ we can assign a (possibly infinite) connected set $V_x \subset V(G)$, called the \emph{branch set} of $x$, so that these sets $V_x$ are pairwise disjoint and $G$ contains a $V_x-V_y$ edge whenever $xy$ is an edge of $H$.}
This is no longer true for arbitrary infinite graphs. Thomas \cite{thomas1988counter} has constructed a sequence $G_1,G_2,\ldots$ of \emph{binary trees with tops} of size of size continuum, such that $G_i \not\preceq G_j$ whenever $i < j$. Here, \emph{binary tree with tops} describes the class of graphs where one selects in the rooted infinite binary tree $T_2$ a collection $\mathcal{R}$ of rays all starting at the root, adds for each $R \in \script{R}$ a new vertex $v_R$, and makes $v_R$ adjacent to all vertices on $R$. Let us write $G(\script{R})$ for the resulting graph. In his proof, Thomas carefully selects continuum-sized collections of rays $\script{R}_1, \script{R}_2, \script{R}_3, \ldots$ such that $G_i = G(\script{R}_i)$ form the desired bad sequence. 

Thomas's result raises the question whether infinite graphs smaller than size continuum are well-quasi ordered. While this question for countable graphs is arguably the most important open problem in infinite graph theory, Komj\'ath \cite{komjath1995note} has established that for all other (uncountable) cardinals $\kappa$, there are in fact $2^\kappa$ pairwise minor-incomparable graphs of size~$\kappa$. 

The purpose of this note is to give an alternative construction for Komj\'ath's result which is simpler than the original, and also more integrated with other problems in the area:

First, our construction reinstates a pleasant similarity to Thomas's original strategy: The desired minor-incomparable graphs can already be found amongst the $\kappa$-regular trees with $\kappa$ many tops. Second, our construction bears a surprising similarity to a family of rays considered in the 60's by A.H.~Stone in his work on Borel isomorphisms \cite{stone1963sigma}. Third, our examples allow for a considerable sharpening of a result by Thomas and Kriz \cite{kriz1990clique} on graphs without uncountable clique minors but arbitrarily large tree width. And finally, a very similar family of graphs had recent applications for results about normal spanning trees in infinite graphs \cite{Pitz2020b}.

\section{Trees with tops and Stone's example}

Consider the order tree $(T,\leq)$ where the nodes of $T$ are all sequences of elements of $\kappa$ of length $\leq \omega$ including the empty sequence, and let $t \leq t'$ if $t$ is a proper initial segment of $t'$. The graph on $T$ where any two comparable vertices are connected by an edge was considered by Kriz and Thomas in \cite{kriz1990clique} where they showed that any tree-decomposition of this graph must have a part of size $\kappa$, despite not containing a subdivision of an uncountable clique. 

For our purposes, however, it suffices to consider a graph $G$ on $T$ such that  any node represented by finite sequences of length $n$ is connected to all its successors of length $n+1$ in the tree order $\leq$, and any node represented by an $\omega$-sequence is connected to all elements below  in the tree order $\leq$. Clearly, $G$ is connected. We later use the simple fact that 

\begin{enumerate}[label=(\roman*)]
\item \label{itemi} every connected subgraph $H \subset G$ has a unique minimal node $t_H$ in  $(T,\leq)$.
\end{enumerate}

Now given a set $S \subset \kappa$ consisting just of cofinality $\omega$ ordinals, choose for each $s \in S$ a cofinal sequence $f_s \colon \omega \to s$, and let $F=F(S) := \set{f_s}:{s \in S}$ be the corresponding collection of sequences in $\kappa$. Let $T^S$ denote the subtree of $T$ given by all finite sequences in $T$ together with $F(S)$, and let $G(S)$ denote the corresponding induced subgraph of $G$. We will refer to $G(S)$ as a `$\kappa$-regular tree with tops', where the elements of $F(S)$ are of course the `tops'.

To the author's best knowledge, such a collection of tree branches $F(S) = \set{f_s}:{s \in S}$ for $S$ the set of  \emph{all} cofinality $\omega$ ordinals was first considered by Stone in \cite[\S5]{stone1963sigma} for the case $\kappa = \omega_1$ and in \cite[\S3.5]{stone1972non} for the general case of uncountable regular $\kappa$.

We consider below graphs $G(S)$ where $S \subset \kappa$ is stationary. Recall that a subset $A \subset \kappa$ is \emph{unbounded} if $\sup A = \kappa$, and \emph{closed} if $\sup (A \cap \ell) = \ell$ implies $\ell \in A$ for all limits $\ell< \kappa$. The set $A$ is a \emph{club} in $\kappa$ if it is both closed and unbounded. A subset $S \subset \kappa$ is \emph{stationary} (in $\kappa$) if $S$ meets every club of $\kappa$. Below, we use the following two elementary properties of stationary sets of regular uncountable cardinals $\kappa$ (for details see e.g.\ \cite[\S8]{jech2013set}):
\begin{enumerate}[label=(\roman*),resume]
\item \label{itemii} If $S \subset \kappa$ is stationary and $S = \bigcup \set{S_n}:{n \in \N}$, then some $S_n$ is stationary.
\item \label{itemiii} \emph{Fodor's lemma}: If $S \subset \kappa$ is stationary and $f \colon S \to \kappa$ is such that $f(s)<s$ for all $s \in S$, then there is $i< \kappa$ such that $f^{-1}(i)$ is stationary. 
\end{enumerate}

\section{Constructing families of minor-incomparable graphs}

At the heart of Komj\'ath's proof lies the construction, for regular uncountable $\kappa$, of $\kappa$ pairwise minor-incomparable connected graphs of cardinality $\kappa$. From this, the singular case follows, and by considering disjoint unions of these graphs, one obtains an antichain of size $2^\kappa$, see \cite[Lemma~2]{komjath1995note}. Hence, it will be enough to prove:

\begin{theorem}
For regular uncountable $\kappa$, the class of $\kappa$-regular trees with $\kappa$ many tops contains a minor-antichain of size $\kappa$.
\end{theorem}

\begin{proof}
As the set of cofinality $\omega$ ordinals of a regular uncountable $\kappa$ splits into $\kappa$ many disjoint stationary subsets \cite[Lemma 8.8]{jech2013set}, it suffices to show: If $S,R$ are disjoint stationary subsets consisting of cofinality $\omega$ ordinals, then $G(S) \not \preceq G(R)$.

Suppose for a contradiction that $G(S) \preceq G(R)$. For ease of notation, we identify $s$ with $f_s$ for all $s \in S$, and similarly for $R$. For $v \in T^S$ write $t_v \in T^{R}$ for the by \ref{itemi} unique minimal node of the branch set of $v$ in $G(R)$. Note that if $v,w$ are adjacent in $G(S)$, then $t_v$ and $t_w$ are comparable in $(T^R,\leq)$. Since $T^R$ has countable height, by \ref{itemii} there is a stationary subset $S' \subset S$ such that all $t_{s}$ for $s \in S'$ belong to the same level of $T^R$. Suppose for a contradiction this level has finite height $n$. By applying Fodor's lemma~\ref{itemiii} iteratively $n+1$ times, we obtain a stationary subset $S'' \subset S'$ such that all $f_s$ for $s \in S''$ agree on $f_s(i)$ for $i \leq n$. So distinct $t_{s}$ for $s \in S''$ have at least $n+1$ common neighbours below them in $(T^R,\leq)$, a contradiction.

Thus, we may assume that $t_{s} \in R$ for all $s \in S$, 
giving rise an injective function $f \colon S \to R, \; s \to t_{s}$.
%
Since $f$ is injective, we cannot have $x< f(x)$ on a stationary subset of $S$ by Fodor's lemma~\ref{itemiii}. Hence, we may further assume that $f(x) \geq x$ for all $x \in S$. 

For $i < \kappa$ let $T^S_i$ be the subtree of $T^S$ of all elements whose coordinates are strictly less than $i$, and consider the function $g \colon \kappa \to \kappa, \; i \mapsto \min \set{j < \kappa}:{t_v \in T^R_j \text{ for all } v \in T^S_i}$. Since $\kappa$ is regular, the function $g$ is well-defined. And clearly, $g$ is increasing.
The function $g$ is also continuous. 
Indeed, for a limit $\ell < \kappa$ consider any 
$v \in T^S_\ell \setminus \bigcup_{i < \ell} T^S_i$.
Clearly, 
 $v$ is a top, and so all its neighbours belong to $\bigcup_{i < \ell} T^S_i$. Hence, $t_v$ must be comparable to infinitely many nodes in $\bigcup_{i < \ell} T^R_{g(i)}$, implying that $t_v \in \bigcup_{i < \ell} T^R_{g(i)}$, too. 
 
 Hence, the set of fixed points $C$ of $g$ forms a club in $\kappa$, see \cite[Exercise 8.1]{jech2013set}. But any $s \in S \cap C$ satisfies $s \leq f(s) \leq g(s) = s$, showing that $s = f(s) \in S \cap R$, a contradiction.
\end{proof}

\section{A sharpening of Kriz and Thomas's result}

Kriz and Thomas have used the graph on $T$ where any two comparable vertices are connected by an edge in \cite[Theorem~4.2]{kriz1990clique} as an example of a graph without a subdivision of an uncountable clique, but where any tree-decomposition must have a part of size $\kappa$. For background on tree-decompositions see \cite[\S12]{bible}.

In this section we establish that if $\kappa$ is regular, then already any graph $G(S)$ from above -- for $S \subset \kappa$ stationary and consisting just of cofinality $\omega$ ordinals -- has the same property that any tree-decomposition of $G$ must have a part of size $\kappa$. 

\begin{theorem}
For regular uncountable $\kappa$ and $S \subset \kappa$ stationary and consisting just of cofinality $\omega$ ordinals, any tree-decomposition of $G(S)$ has a part of size $\kappa$.
\end{theorem}

We remark that this result is sharp: If $\kappa =\omega_1$ and $S \subset \omega_1$ is non-stationary, then $G(S)$ has a normal spanning tree \cite{Pitz2020b}, and hence a tree-decomposition into finite parts.

In the proof, $Q^n$ denotes the $n$th level of a rooted tree $Q$, and $Q^{<n} = \bigcup_{m<n} Q^m$.

\begin{proof}
We start with an observation: Given any stationary subset $S' \subset S$ and any set $X$ of vertices of $G(S)$ with $|X|<\kappa$, at least one component of $G(S) - X$ contains a stationary subset of $S'$. Indeed, using the notation as in the previous proof, since $X$ is small,  we have $X \subset T^S_i$ for some $i < \kappa$, and $S' \setminus T^S_i$ is still stationary. Now for every $s \in S' \setminus T^S_i$, let $n_s$ be minimal such that $f_s(n_s) \notin T^S_i$. By \ref{itemii} there is a stationary subset $S'' \subset S' \setminus T^S_i$ whose elements agree on $n = n_s$. By applying Fodor's lemma~\ref{itemiii} $n+1$ times, there is a stationary subset $S''' \subset S''$ such that the maps $f_s$ agree on their first coordinates up to $n$ for all $s \in S'''$. Then all 
$s \in S'''$ have a common neighbour outside of $T^S_i$, so belong to the same component. 

Now suppose for a contradiction that $G(S)$ has a tree-decomposition $(Q,(V_q)_{q \in Q})$ such that $|V_q|< \kappa$ for all $q \in Q$. Fix an arbitrary root $q_0$ of $Q$. Every $s \in S$ is contained in some part; let $q_s \in Q$ be minimal in the tree order of $Q$ such that $s \in V_{q_s}$. By \ref{itemii} there is a minimal $n \in \N$ and a stationary subset $S' \subset S$ such that $q_s \in Q^n$ for all $s \in S'$. 

Then no component of $G - \bigcup_{q \in Q^{{<}n}} V_q$ contains a stationary subset of $S'$, as the intersection of any such component with $S'$ is contained in some less than $\kappa$ sized $V_{q_s}$. 
However, by iteratively using our earlier observation, there is a path $q_0q_1\ldots q_{n-1}$ starting at the root of $Q$, a decreasing sequence of components $C_i$ of $G - \p{V_{q_0} \cup \cdots \cup V_{q_i}}$ and stationary subsets $S_i$ of $S'$ such that each $S_i \subset C_i$. But $S_{n-1} \subset C_{n-1} \subset G - \bigcup_{q \in Q^{{<}n}} V_q$,  a contradiction.
\end{proof}

\bibliographystyle{plain}
\bibliography{reference}

\begin{thebibliography}{1}

\bibitem{bible}
R.~Diestel.
\newblock {\em {Graph Theory}}.
\newblock Springer, 5th edition, 2015.

\bibitem{jech2013set}
Thomas Jech.
\newblock {\em Set theory, The Third Millennium Edition}.
\newblock Springer Monographs in Mathematics, 2013.

\bibitem{komjath1995note}
P{\'e}ter Komj{\'a}th.
\newblock A note on minors of uncountable graphs.
\newblock In {\em Mathematical Proceedings of the Cambridge Philosophical
  Society}, volume 117, pages 7--9, 1995.

\bibitem{kriz1990clique}
Igor Kriz and Robin Thomas.
\newblock Clique-sums, tree-decompositions and compactness.
\newblock {\em Discrete mathematics}, 81(2):177--185, 1990.

\bibitem{Pitz2020b}
Max Pitz.
\newblock A new obstruction for normal spanning trees.

\bibitem{stone1963sigma}
Arthur~H. Stone.
\newblock On $\sigma$-discreteness and {B}orel isomorphism.
\newblock {\em American Journal of Mathematics}, 85(4):655--666, 1963.

\bibitem{stone1972non}
Arthur~H. Stone.
\newblock Non-separable {B}orel sets, {II}.
\newblock {\em General Topology and its Applications}, 2(3):249--270, 1972.

\bibitem{thomas1988counter}
Robin Thomas.
\newblock A counter-example to `{W}agner's conjecture' for infinite graphs.
\newblock In {\em Mathematical Proceedings of the Cambridge Philosophical
  Society}, volume 103, pages 55--57, 1988.

\end{thebibliography}

\end{document}